\documentclass[12pt,twoside]{amsart}
\usepackage{amsmath}
\usepackage{amsthm}
\usepackage{amsfonts}
\usepackage{amssymb}
\usepackage{latexsym}
\usepackage{mathrsfs}
\usepackage{amsmath}
\usepackage{amsthm}
\usepackage{amsfonts}
\usepackage{amssymb}
\usepackage{latexsym}
\usepackage{geometry}
\usepackage{dsfont}
\usepackage[dvips]{graphicx}
\usepackage{color}
\usepackage[all]{xy}

\date{}
\allowdisplaybreaks[4] \footskip=15pt
\renewcommand{\uppercasenonmath}[1]{}

\usepackage{graphicx,amssymb}
\usepackage[all]{xy}
\usepackage{amsmath}

\allowdisplaybreaks
\usepackage{amsthm}
\usepackage{color}

\theoremstyle{plain}
\newtheorem{theorem}{Theorem}[section]

\newtheorem{corollary}[theorem]{Corollary}

\newtheorem*{open question}{Open Question}

\theoremstyle{definition}
\newtheorem*{acknowledgement}{Acknowledgement}

\theoremstyle{remark}

\def\p{\frak p}

\begin{document}
\begin{center}
{\large  \bf On Cohen's theorem for  Artinian modules}

\vspace{0.5cm}   Xiaolei Zhang$^{a}$,\  Hwankoo Kim$^{b}$,\ Wei Qi$^{c}$

{\footnotesize a.\ School of Mathematics and Statistics, Shandong University of Technology, Zibo 255049, China\\
b.\ Division of Computer and Information Engineering, Hoseo University, Asan 31499, Republic of Korea\\
c.\  School of Mathematical Sciences, Sichuan Normal University, Chengdu 610068,  China\\

E-mail: hkkim@hoseo.edu\\}
\end{center}

\bigskip
\centerline { \bf  Abstract}
\bigskip
\leftskip10truemm \rightskip10truemm \noindent

In this paper, we  prove that a finitely embedded  $R$-module $M$ is Artinian if and only if for every prime ideal $\p$ of $R$ with $(0:_RM)\subseteq \p$, there exists a submodule $N^\p$ of $M$  such that $M/N^\p$ is finitely  embedded and $M[\p]\subseteq N^\p\subseteq (0:_M\p)$.
\vbox to 0.3cm{}

\leftskip0truemm \rightskip0truemm
\bigskip

\section{Introduction}
Throughout this article, all rings are commutative rings with identity and modules are unitary.

It is well-known that Cohen's Theorem states that a  ring $R$ is a Noetherian ring  if and only if every prime ideal of $R$ is finitely generated (see \cite[Theorem 2]{c50}). In 1994, Smith  extended  Cohen's Theorem from rings to  modules, which states that a finitely generated $R$-module $M$ is Noetherian if and only if the submodules $\p M$ of $M$ are finitely generated for every prime ideal $\p$ of $R$, if and only if $M(\p)$ is finitely generated  for each prime ideal $\p$ of $R$ with $(0:_RM)\subseteq\p$, where $M(\p)=\{x\in M\mid sx\in \p M $ for some $s\in R-\p \}$ (see \cite{s94}).  Very recently, Parkash and  Kour \cite[Theorem 2.1]{pk21} generalized the Smith's result on Noetherian modules and obtained that  a finitely generated $R$-module $M$ is Noetherian if and only if for every prime ideal $\p$ of $R$ with $(0:_RM)\subseteq \p$, there exists a finitely generated  submodule $N_\p$ of $M$ such that $\p M\subseteq N_\p\subseteq M(\p)$.

The main motivation of this paper is to  dualize  Parkash and  Kour's results to Artinian modules. We recall some basic notions on finitely embedded modules and Artinian modules (refer to \cite{sv72} for example). Let  $R$ be a ring and $M$ an $R$-module. $M$ is said to be finitely embedded if there exists finitely many simple modules $S_1,S_2,\cdots,S_n$ such that $E(M)$ is isomorphic to $E(S_1)\oplus E(S_2)\oplus\cdots \oplus E(S_n)$, where $E(M)$ and $E(S_k)$ are the injective envelopes of $M$ and $S_k$ respectively. The class of finitely embedded modules is closed under submodules and extensions by \cite[Proposition 3.20]{sv72}. A family $\{M_i\}_{i\in \Lambda}$ of submodules of $M$ is said to be an inverse system if for any finite number of $i_1,i_2,\cdots,i_k$ of $\Lambda$, there is an element $i\in \Lambda$ such that $M_i\subseteq \bigcap\limits_{k=1}^nM_{i_k}$. By \cite[Proposition 3.19]{sv72}, $M$ is finitely embedded if and only if every inverse system of nonzero submodules of $M$ is bounded below by a nonzero submodule of $M$.  $M$ is said to be  Artinian if it satisfies the minimal condition for submodules, or equivalently, the descending chain condition for submodules. It is well known a Noetherian module is exactly a module in which all submodules are finitely generated. Dually,  $M$ is  Artinian if and only if every factor module of $M$ is finitely embedded (see  \cite[Theorem 3.21]{sv72}). In 2006, Nishitani studied Cohen's Theorem for Artinian modules and showed that a finitely embedded module $M$ is Artinian if and only if $M/(0:_M\p)$ is finitely  embedded  for every prime ideal $\p$ of $R$. We have generalized the Nishitani's result in Theorem \ref{main}, which can also be seen as a dualization of Parkash and  Kour's results.

\section{Results}

Let $R$ be a ring, $\p$ be a prime ideal of $R$ and $M$ an $R$-module. Define $M[\p]=\bigcap\limits_{s\in R-\p}s(0:_M\p)$. Then $M[\p]$ is obviously a submodule of $M$.

\begin{theorem}\label{main} Let $R$ be a ring and $M$  a finitely embedded  $R$-module. Then $M$ is Artinian if and only if for every prime ideal $\p$ of $R$ with $(0:_RM)\subseteq \p$, there exists a submodule $N^\p$ of $M$  such that $M/N^\p$ is finitely  embedded and $M[\p]\subseteq N^\p\subseteq (0:_M\p)$.
\end{theorem}
\begin{proof}Suppose $M$ is an Artinian $R$-module and $\p$ is  a prime ideal with $(0:_RM)\subseteq \p$.  If we  take $N^\p=(0:_M\p)$, then $N^\p$ is certainly a submodule of $M$  such that $M/N^\p$ is finitely  embedded and $M[\p]\subseteq N^\p\subseteq (0:_M\p)$ by \cite[Theorem 3.21]{sv72}.

Conversely, suppose that $M$ is not Artinian. Then there exists  a submodule $N'$ of $M$ such that $M/N'$ is not finitely embedded by \cite[Theorem 3.21]{sv72}. Consider the set $\Gamma:=\{N\leq N'\mid M/N$ is not finitely embedded$\}$. Then $\Gamma$ is non-empty as $N'\in \Gamma$. Make a partial order on $\Gamma$ by the opposite of inclusion, that is, $N_1\geq N_2$ if and only if $N_1\subseteq N_2$ in $\Gamma$.

{\bf Claim 1: There exists a maximal element $N\in \Gamma$.} Let $\{N_i\mid i\in \Lambda\}$ be a total ordered subset of $\Gamma$. Set $N=\bigcap\limits_{i\in \Lambda}N_i$. Then $M/N$ is not finitely embedded. Indeed, since $\{N_j/N\}_{j\in\Lambda}$ is an  inverse system of submodules of $M/N$ and there is no nonzero submodule of $M/N$ which is contained in each $N_j/N$. By \cite[Proposition 3.19]{sv72}, there are two possibilities: either $N_j/N=0$ for some $j\in \Lambda$, or $M/N$ is not finitely embedded. In the former case, $N=N_j$ and thus $M/N$ is is not finitely embedded in both cases. Consequently, by Zorn's Lemma, $\Gamma$ has a maximal element, which is also denoted by $N$.  Set $\p=(0:_RN)$.

{\bf Claim 2: $\p$ is a prime ideal.} Indeed, let $a\not\in \p, b\not\in \p$ be elements in $R$. Then $(0:_Na)\subsetneq N$.  Thus $M/(0:_Na)$ is finitely embedded, and so is $(0:_Ma)/(0:_Na)$.  Consider the exact sequence $0\rightarrow (0:_Ma)/(0:_Na)\rightarrow M/N\rightarrow aM/aN\rightarrow 0$. We have $aM/aN$ is not finitely embedded. Thus $M/aN$ is  not finitely embedded. So $aN=N$ by the maximality of $N$. Similarly, $bN=N$. Hence $abN=N\not=0$ as $M$ is  finitely embedded. So $ab\not\in\p$.

{\bf Claim 3: $N\subseteq M[\p]$.} Indeed, suppose there is  $y\in N$ such that $y\not\in  M[\p]$. Then $y\not\in s(0:_M\p)$ for some $s\in R-\p$. Since $N\subseteq (0:_M\p)$, we have $sN\subsetneq N$.  Hence $M/sN$ is  finitely embedded. Since $s\not\in\p$, we have $(0:_Ns)\subsetneq N$.  So $M/ (0:_Ns)$ is finitely  embedded. Consider the exact sequence $$0\rightarrow (0:_Ms)/ (0:_Ns)\rightarrow M/N\rightarrow sM/sN\rightarrow 0.$$
Since $M/sN$ is  finitely embedded, the submodule $sM/sN$ is also finitely embedded. Since $M/(0:_Ns)$ is finitely embedded, the submodule $(0:_Ms)/ (0:_Ns)$ is also finitely embedded. Hence $M/N$ is finitely embedded, which is a contradiction.

Now, we will show  $M$ is Artinian. Suppose the finitely embedded  $R$-module  $M$ is not Artinian, then there is an ideal $I$ of $R$ such that $(0:_MI)$ is Artinian and $M/(0:_MI)$ is not finitely embedded by  \cite[Lemma 7]{N06}. Furthermore, there is a submodule  $N$ of $(0:_MI)$ such that $M/N$ is not finitely embedded and $\p=(0:_RN)$ is prime by Claim 1 and Claim 2. Since $N\subseteq (0:_MI)$, we have  $(0:_M\p)\subseteq (0:_MI)$. Thus the quotient $(0:_M\p)/N$ is Artinian, and thus is finitely embedded. Since  $(0:_RM)\subseteq \p$, there is a submodule $N^\p$ of $M$  such that $M/N^\p$ is finitely  embedded and $N\subseteq M[\p]\subseteq N^\p\subseteq (0:_M\p)$ by assumption and Claim 3. And then the submodule $N^{\p}/N$ of $(0:_M\p)/N$ is finitely embedded. Consider the following exact sequence $$0\rightarrow N^{\p}/N\rightarrow M/N\rightarrow M/N^{\p}\rightarrow 0.$$
Since $M/N^{\p}$ and $N^{\p}/N$ are finitely  embedded, $ M/N$ is also  finitely  embedded, which is a contradiction. Hence $M$ is Artinian.
\end{proof}

\begin{corollary} Let $R$ be a ring. A finitely embedded  $R$-module $M$ is Artinian if and only if $M/(0:_M\p)$ is finitely  embedded  for every prime ideal $\p$ of $R$ with $(0:_RM)\subseteq \p$.
\end{corollary}

\begin{corollary} Let $R$ be a ring. A finitely embedded  $R$-module $M$ is Artinian if and only if $M/M[\p]$ is finitely  embedded  for every prime ideal $\p$ of $R$ with $(0:_RM)\subseteq \p$.
\end{corollary}

\begin{acknowledgement}\quad\\
The first author  was supported by the National Natural Science Foundation of China (No. 12061001).
\end{acknowledgement}

\end{document}